\documentclass{amsart}
\usepackage{amsmath,amsthm,amssymb,xypic,array}
\usepackage[all]{xy}
\usepackage{hyperref}

\theoremstyle{plain}

\newtheorem{theorem}{Theorem}[section]
\newtheorem{theoremn}{Theorem}
\newtheorem{proposition}[theorem]{Proposition}
\newtheorem{lemma}[theorem]{Lemma}
\newtheorem{corollary}[theorem]{Corollary}

\newtheorem{remark}[theorem]{Remark}

\theoremstyle{definition}

\newtheorem{definition}[theorem]{Definition}

\theoremstyle{remark}

\newtheorem{claim}{Claim}

\DeclareMathOperator{\Bl}{Bs}

\DeclareMathOperator{\cod}{cod}
\DeclareMathOperator{\mult}{mult}

\DeclareMathOperator{\Aut}{Aut}
\DeclareMathOperator{\Sing}{Sing}

\newcommand{\QED}{\ifhmode\unskip\nobreak\fi\quad {\rm Q.E.D.}} 

\newcommand\map{\dasharrow}
\newcommand\iso{\cong}

\newcommand{\f}{\varphi}

\renewcommand{\H}{\mathcal{H}}

\newcommand{\I}{\mathcal{I}}
\newcommand{\K}{\mathcal{K}}
\renewcommand{\L}{\mathcal{L}}
\newcommand{\M}{\mathcal{M}}

\newcommand\Mo[1]{\overline{M}_{0,#1}}
\renewcommand{\O}{\mathcal{O}}
\renewcommand{\P}{\mathbb{P}}

\newcommand{\R}{\mathbb{R}}

\newcommand{\rat}{\dasharrow}

\begin{document}
\title{The automorphisms group of $\Mo{n}$}

\author{Andrea Bruno and
Massimiliano Mella}
\address{ Dipartimento di Matematica \\
Universit\`a Roma Tre \\L.go S.L.Murialdo, 1 \\
 \newline
\indent 00146 Roma Italia}
\address{Dipartimento di Matematica\\
Universit\`a di Ferrara\\
Via Machiavelli 35\\
 \newline
\indent 44100 Ferrara Italia}
\email{bruno@mat.uniroma3.it}
\email{mll@unife.it}

\date{May 2010}
\thanks{Partially supported by Progetto PRIN 2008 ``Geometria
  sulle variet\`a algebriche'' MUR}
\maketitle


\section*{Introduction}
The moduli space $M_{g,n}$ of smooth
n-pointed curves of genus $g$,
and its projective closure, the Deligne-Mumford
compactification $\overline{M}_{g,n}$,
is a classical object of study that reflects
many of the properties of
families of pointed curves. As a matter of fact, the
study of its biregular geometry is of interest
in itself and has become a central theme in various
areas of mathematics.

Already for small $n$, the moduli spaces
$\Mo{n}$ are quite intricate objects deeply rooted
in classical algebraic geometry. Under this perspective,
Kapranov showed in \cite{Ka} that $\Mo{n}$ is identified with the
closure of the subscheme of the Hilbert scheme parametrizing
rational normal curves passing through
$n$  points in linearly general position in $\P^{n-2}$.
Via this identification, given $n-1$ points in linearly
general position in $\P^{n-3}$, $\Mo{n}$ is isomorphic
to an iterated blow-up of $\P^{n-3}$ at the strict
transforms of all the
linear spaces spanned by subsets of the points in order
of increasing dimension.
In a natural way, then, base point free linear systems on
$\Mo{n}$ are identified
with linear systems on $\P^{n-3}$ whose base locus is
quite special and supported on
so-called vital spaces, i.e. spans of subsets of the given points.
Another feature of this picture is that all these
vital spaces correspond to divisors in $\Mo{n}$ which have
a modular interpretation as products of $\overline{M}_{0,r}$ for $r<n$.
In this interpretation, the modular forgetful maps
$\phi_I: \Mo{n} \to \overline{M}_{0,n-|I|}$, which forget points
indexed by $I \subset \{ 1, \ldots, n \}$, correspond, up to
standard Cremona transformations, to linear projections from vital spaces.
The aim of this paper is to study automorphisms of $\Mo{n}$
with the aid of Kapranov's beautiful description. 

It is expected that the only possible biregular automorphisms of $\Mo{n}$ are the one 
associated to a permutation of the markings. Any such morphism has to permute the forgetful maps
onto $\Mo{n-1}$ as well. This induces, on 
$\P^{n-3}$, special birational maps that switch lines through $n-1$ points in general position.
  On the other hand if we
 were able to prove that any automorphism has to permute forgetful maps this should lead to a
proof that every automorphism is a permutation.
Our main tool to classify $Aut(\Mo{n})$ is therefore the following Theorem.
 
\begin{theoremn}
Let
  $f:\overline{M}_{0,n}\to
  \overline{M}_{0,r_1}\times\ldots\times\overline{M}_{0,r_h}$ be a dominant morphism with connected fibers. Then  $f$
  is a forgetful map.
\end{theoremn}

The above Theorem is an easy extension of the same statement with one factor only and the latter is
obtained via an inductive argument starting from the case of a morphism with connected fibers onto $\P^1$. 

\begin{theoremn}
Any dominant morphism with connected fibers
  $f:\overline{M}_{0,n}\to \Mo{4}\iso\P^1$
  is a forgetful map.
\end{theoremn}

The idea of proof is as follows. Any morphism of this type produces a pencil of hypersurfaces on $\P^{n-3}$.
The base locus of this pencil has severe geometric restrictions coming from Kapranov's construction.
These are enough to prove that up to a standard Cremona transformation any such pencil is a pencil of hyperplanes.

As already observed this, together with some computation on certain birational endomorphisms of $\P^{n-3}$, is enough to describe the automorphisms group
of $\Mo{n}$.

\begin{theoremn} Assume that $n\geq 5$,
then $\Aut(\Mo{n})=S_n$, the
symmetric group on n elements.
\end{theoremn}

This
result has a natural counterpart in 
the Teichm\"uller-theoretic literature on 
the automorphisms of moduli spaces
$M_{g,n}$ developed in a series of
papers by  Royden, Earle--Kra, and
others, \cite{Ro}, \cite{EK},
\cite{Ko}, but we do not see a
straightforward way to go from one to the
other. 

 In this paper we study
``modular'' fiber type morphisms
on $\Mo{n}$ via the study of linear systems on
$\P^{n-3}$ and applying whenever possible classical
projective techniques. This program has been
recently pursued also by  Bolognesi, \cite{Bo},
in his description of birational models of
$\Mo{n}$. In a forthcoming paper,
\cite{BM}, we plan to study fiber type
morphisms of $\Mo{n}$ onto either low dimensional
varieties or with low $n$ or
with   linear general fiber.

The second named author would like to
thank Gavril Farkas for rising his
attention on Kapranov's paper,
\cite{Ka}, and the possibility to use
projective techniques in the study of $\Mo{n}$.
It is a pleasure to thank Gavril Farkas,
Angela Gibney,
 and Daniel Krashen, for their
enthusiasm  in a preliminary version of
these papers and for their suggestions to
clarify the contents, many thanks are
also due to Brendan Hassett for
Teichm\"uller bibliography and James M\textsuperscript{c}Kernan
and Jenia Tevelev for the nice
alternative proof of Theorem \ref{thm:p1}.

\section{Preliminaries}
\label{sec:notation}
We work over the field of complex numbers. An n-pointed curve
of arithmetic genus $0$ is the datum $(C; q_1, \ldots, q_n)$ of
a tree of smooth rational curves and $n$ ordered points on the
nonsingular locus of $C$ such that each component of $C$
has at least three points which are either marked or singular
points of $C$. If $n \geq 3$, $\Mo{n}$ is the smooth
$(n-3)$-dimensional scheme
constructed by Deligne-Mumford, which is the fine moduli scheme
of isomorphism classes $[(C; q_1, \ldots, q_n)]$ of stable n-pointed curves
of arithmetic genus $0$.

\noindent For any $ i \in \{1, \ldots, n \}$ the forgetful map
$$\phi_i : \Mo{n} \to \overline{M}_{0,n-1}$$
is the surjective morphism which associates
to the isomorphism class $[(C; q_1, \ldots, q_n)]$ of
a stable n-pointed rational curve $(C; q_1, \ldots, q_n)$
the isomorphism class of the $(n-1)$-pointed stable
rational curve obtained by forgetting $q_i$ and, if any,
contracting to a point a component of $C$ containing
only $q_i$, one node of $C$ and another marked point, say $q_j$.
The locus of such curves forms a divisor,
that we will denote by $E_{i,j}$. The morphism $\phi_i$ also plays the
role of the universal curve morphism, so that its
fibers are all rational curves transverse to
$n-1$ divisors $ E_{i,j}$.
Divisors $E_{i,j}$ are the images
of $n-1$ sections $s_{i,j}: \Mo{n-1} \to \Mo{n}$
of $\phi_i$. The section $s_{i,j}$ associates to
$[(C; q_1, \ldots,q_i^{\vee}, \ldots,  q_n)]$
the isomorphism class of the n-pointed stable
rational curve obtained by adding at $q_j$ a smooth rational curve
with marking of two points, labelled by $q_i$ and $q_j$.
Analogously, for every $I \subset \{1, \ldots, n \}$,
we have well defined forgetful
maps $\phi_I : \Mo{n} \to \overline{M}_{0,n-|I|}$.
>From our point of view the important part of a forgetful map is the set of forgotten index, more than the actual marking of the remaining. For this we slightly abuse the language and introduce the following definition.
\begin{definition}
\label{def:forg} A forgetful map is the composition of $\phi_I:\Mo{n}\to \Mo{r}$ with an automorphism $g\in Aut(\Mo{r})$ that permutes the markings.
\end{definition}
\noindent In order to avoid trivial cases we will always
tacitly consider $\phi_I$ only if $n-|I| \geq 4$.

Besides the canonical class $K_{\Mo{n}}$, on
$\Mo{n}$ are defined line bundles $\Psi_i$ for each $i \in \{1, \ldots, n \}$ as follows:
the fiber of $\Psi_i$ at a point $[(C; q_1, \ldots, q_n)]$ is the tangent line $T_{C,p_i}$.
Kapranov, in \cite{Ka} proves the following:

\begin{theorem}
Let $p_1,\ldots,p_n \in \P^{n-2}$ be points in linear
general position. Let $H_n$ be the Hilbert scheme
of rational curves of degree $n-2$ in $\P^{n-2}$.
$\Mo{n}$ is isomorphic
to the subscheme $H \subset H_n$ parametrizing
curves containing $p_1, \ldots, p_n$.
For each $i \in  \{1,\ldots,n\}$ the line bundle
$\Psi_i$ is big and globally generated
and it induces a morphism $f_i : \Mo{n} \to \P^{n-3}$
which is an iterated blow-up
of the projections from $p_i$ of the given points
and of all strict transforms of the
linear spaces they generate, in order of increasing dimension.
\label{kapranov}
\end{theorem}

We will use the following notations:

\begin{definition} A Kapranov set
  $\K\subset\P^{n-3}$ is an ordered  set of $(n-1)$ points
  in linear general position, labelled
  by a subset of $\{1,\ldots, n\}$.
  For any $J \subset \K$, the linear span
  of points in $J$ is said a vital linear subspace
  of $\P^{n-3}$. A vital cycle is any union of vital
  linear subspaces.
\end{definition}

To any Kapranov set, labelled by
$\{1,\ldots, i-1, i+1,\ldots, n\}$,
is uniquely associated a Kapranov map,
$f_i:\Mo{n}\to\P^{n-3}$, with $\Psi_i= f_i^{*} \O_{\P^{n-3}}(1)$,
and to a Kapranov map is uniquely associated
a Kapranov set up to projectivity.

\begin{definition} Given a subset
  $ I= \{i, i_1,\ldots,i_s\} \subset \{1,\ldots,n\}$
  and the Kapranov map
  $f_i:\Mo{n}\to\P^{n-3}$, let
  $I^*=\{1,\ldots,n\}\setminus I$. Then we indicate
  with
$$H^
{i\vee}_{I^*}:=:V^i_{I\setminus\{i\}}:=:V^i_{i_1,\ldots,i_s}:=\langle p_{i_1},\ldots,
  p_{i_s}\rangle \subset \P^{n-3}$$ the vital linear subspace
  generated by the $p_{i_j}$'s and with
$$E_I:=E_{i,i_1,\ldots,i_s}:=f_i^{-1}(V^i_I)$$ the divisor associated on
  $\Mo{n}$.
\end{definition}
Notice that  $H^{h \vee}_{ij}$ is the
hyperplane missing the points $p_i$
and $p_j$ and the set
$$\K'=\K\setminus\{p_i,p_j\}\cup(
H^{\vee}_{ij}\cap\langle p_i,p_j\rangle)$$
is a Kapranov set in $H^{h \vee}_{ij}$.

 In particular for any $i \in \{1,\ldots,n\}$ and Kapranov
set $\K=\{p_1, \ldots, p_i^{\vee}, \ldots, p_n\}$ divisors
$E_{i,j}=f_{i}^{-1}(p_j)$ are defined and such notation is compatible
with the one adopted for the sections $E_{i,j}$ of $\phi_i$.
More generally, for any $i \in I \subset \{1, \ldots, n \}$ the divisor $E_I$
has the following property: its general
point corresponds to the isomorphism class of a rational
curve with two components, one with
$|I|+1$ marked points, the other with $|I^{*}|+1$ marked points,
glued together at the points not marked by elements of
$\{1,\ldots,n\}$.
It follows from this picture that
$E_I = E_{I^*}$ and that $E_I$ is abstractly isomorphic
to $\overline{M}_{0,|I|+1} \times \overline{M}_{0,|I^{*}|+1}$. The divisors
$E_I$ parametrise singular rational curves, and
they are usually called boundary divisors.
A further property of $E_I$ is that for each choice
of $i \in I, j \in I^{*}$, $E_I$ is a section of the forgetful morphism:
$$ \phi_{I \setminus \{ i \} } \times \phi_{I^{*} \setminus \{ j \} } :
 \Mo{n} \to \overline{M}_{0,|I^{*}|+1} \times \overline{M}_{0,|I|+1}.$$
This morphism is surjective and all fibers are rational curves.
With our notations $f_i(E_I)$ is a vital linear space of dimension $|I|-2$
if $i \in I$ and a vital linear space of dimension $|I^{*}|-2$ if $i \notin I$.

\begin{definition} A dominant morphism $f:X\to Y$ is called a fiber type morphism if the dimension of the general fiber
is positive, i.e. $\dim X>\dim Y$.
\end{definition}

We are interested in describing linear systems on $\P^{n-3}$ that are
associated to fiber type morphisms on
$\Mo{n}$. For this purpose we introduce
some definitions.
\begin{definition}\label{def:s} A
  linear system  on a smooth projective variety
  $X$ is uniquely determined
  by a pair $(L,V)$, where $L\in Pic(X)$
  is a line bundle and $V\subseteq
  H^0(X,L)$ is a vector space. If no
  confusion is likely to arise we will
  forget about $V$ and let
  $\L=(L,V)$.  Let
  $g:Y\to X$ be a birational morphism
  between smooth varieties. Let $A\in\L$ be a general
  element and $A_Y=g^{-1}_*A$ the strict
  transform. Then  $g^*L=A_Y+\Delta$ for
  some  effective  $g$-exceptional
  divisor $\Delta$.  The strict transform of
  $\L=(L,V)$ via $g$ is
$$g^{-1}_*\L:=(g^*L-\Delta,V_Y)$$
where $V_Y$ is the vector space spanned by the
strict transform of elements in $V$.
\end{definition}

\begin{definition}\label{def:Mk}  Let $\K\subset\P^{n-3}$
  be a Kapranov set and
  $f_i:\overline{M}_{0,n}\to\P^{n-3}$
  the associated map. An $M_{\K}$-linear system
  on $\P^{n-3}$    is a linear system
  $\L\subseteq|\O_{\P^{n-3}}(d)|$, for some $d$, such that
  $f_{i*}^{-1}\L$ is a base
  point free linear system.
 \end{definition}

Let $\L$ be an $M_\K$-linear system, and fix
$f_1:\Mo{n}\to\P^{n-3}$ a Kapranov
map. To better understand the properties
of $M_\K$-linear systems let us look
closer at $f_1$.  Let
$\epsilon:Y\to\P^{n-3}$ be the blow
up of $p_2\in\K$ with exceptional
divisor $E$. Then Kapranov's map $f_1$,
can be factored as follows

\[
\xymatrix{
\overline{M}_{0,n}\ar[d]^{f_{1}}\ar[r]^{g}
&Y\ar[dl]^{\epsilon} \\
\P^{n-3}&}
\]
with a birational morphism $g:\Mo{n}\to
Y$. The map $g$ is obtained by blowing
up, in the prescribed order, the strict
transform of every vital cycle of
codimension at least 2 in $Y$. in
particular it  is an
isomorphism on every codimension 1 point
of $E\subset Y$. With this
observation we are able to weakly control the
base locus of $\L_Y:=\epsilon_{*}^{-1}\L$, the
strict transform linear system.

\begin{lemma}
 Let $\L$ be an $M_\K$-linear
system without fixed components,
associated to the Kapranov map
$f_i:\Mo{n}\to\P^{n-3}$. Let
$\epsilon:Y\to\P^{n-3}$ be the blow up
of the Kapranov point $p_j\in\K$ with
exceptional divisor $E$. Let
$\L_Y$ be the strict transform. Then the
linear system $\L_{Y|E}$
has not fixed components.
\label{lem:bs}
\end{lemma}
\begin{proof} Let
\[
\xymatrix{
\overline{M}_{0,n}\ar[d]^{f_{i}}\ar[r]^{g}
&Y\ar[dl]^{\epsilon} \\
\P^{n-3}&}
\]
be the commutative diagram as above, with
exceptional divisor $E\subset Y$.
We noticed that $g$ is an isomorphism on
every codimension 1 point of $E$. By
hypothesis $\L$ and hence $\L_Y$ have
not fixed components. By
construction $g$ is a resolution of
$\Bl\L_Y$. This yields that
$$\cod_{E}(\Bl\L_{Y}\cap E)\geq 2,$$
and $\L_{Y|E}$ has not fixed components.
\end{proof}

A further property inherited from Kapranov's
construction is the following.

\begin{remark} \label{rem:ka}
Let $H^{h \vee}_{ij}$ be the
hyperplane missing the points $p_i$
and $p_j$ and $$\K'=\K\setminus\{p_i,p_j\}\cup
(H^{\vee}_{ij}\cap\langle p_i,p_j\rangle)$$
the associated Kapranov set.
Then $\L_{|H^{\vee}_{ij}}$
is an $M_{\K'}$ linear system.
\end{remark}

Basic examples of fiber type morphisms are forgetful maps.
Consider any set $ I \subset \{1, \ldots ,n \} $ and the associated
forgetful map $\phi_I$. If $j \notin I$ a typical diagram we will consider
is
\[
\xymatrix{
\overline{M}_{0,n}\ar[d]^{f_{j}}\ar[r]^{\phi_I}
&\overline{M}_{0,n-|I|}\ar[d]^{f_{h}} \\
\P^{n-3}\ar@{.>}[r]^{\pi_{I}}&\P^{n-|I|}}
\]
where the $f_j$ and $f_h$ are Kapranov maps and
$\pi_{I}$ is the projection from $V^j_{I }$.
In this case an $M_\K$-linear system
associated to $\Phi_I$ is given by
 $|\O(1)\otimes\I_{V^j_{I }}|$
and if $F_I$ is any fiber of $\phi_I$,
$f_j(F_I)$ is a linear space of
codimension $|I|$.

It is important, for what follows, to
understand explicitly the rational map $\pi_I$ when
$j\in I$. To do this we use Cremona
transformations.
\section{Cremona transformations and
  $M_\K$-linear systems}
\label{sec:bir}
\begin{definition}\label{def:cre} Let $\K$ be a
  Kapranov set with Kapranov map
  $f_i:\Mo{n}\to\P^{n-3}$. Then let
$$\omega^\K_j:\P^{n-3}\rat\P^{n-3}$$
 be the standard Cremona transformation centered on
  $\K\setminus \{p_j\}$. Via Kapranov's
construction we can associate a Kapranov
set  labelled by
$\{1,\ldots,n\}\setminus\{j\}$ to the
rhs $\P^{n-3}$ and in this notation
Kapranov, \cite[Proposition 2.14]{Ka} proved that
$$\omega_j^\K= f_j\circ f_i^{-1}$$ as
birational maps.
By a slight abuse of notation we can define
$\omega_j^\K(V^i_I):=f_j(E_{I,i})$,
even if $\omega_j$ is not defined on the
general point of $V^i_I$.
\end{definition}

\begin{remark}\label{rem:kcre}
Let $\omega^\K_h$ be the standard
Cremona transformation centered on
$\K\setminus\{p_h\}$, and
$\K'$ the Kapranov set associated to the hyperplane $H^{i
    \vee}_{jk}$. Then for $h\neq
j,k$ we have  $\omega^\K_{h|H^{i
    \vee}_{jk}}=\omega^{\K'}_h$. This
extends to arbitrary vital linear spaces.
It follows from the definitions that
$\omega_h^\K(V^i_I)=V^h_{I \setminus \{h\},i }$
if $h \in I$ and $\omega_h^\K(V^i_I)=V^h_{(I\cup\{i\})^{*}
\setminus \{h\}}$ if $h \in I^{*}$.
\end{remark}

Let us start with the special case of
forgetful maps onto $\Mo{4}\iso\P^1$.
Let
$$\phi_I:\overline{M}_{0,n}\to
  \P^1\iso\overline{M}_{0,4}$$ be a
  forgetful map and
  \hbox{$\M=\phi_I^*\O(1)$.} Choose a Kapranov
  map
  $f_i:\overline{M}_{0,n}\to\P^{n-3}$
  with
  $\L:=f_{i*}\M\subset|\O(1)|$. As
  already noticed this is
  equivalent to choose $i\not\in I$.
Then $\Bl\L=P$ is a codimension 2 linear
space and we may assume, after
reordering the indexes, that
$$\K\setminus(\Bl\L\cap
\K)=\{p_1,p_2,p_3\} \mbox{ and $i=4$}.$$

To understand what is the linear system
$\L_5:=f_{5*}\M$ we use Kapranov
description of the map
$\omega^\K_5$, see Definition
\ref{def:cre}.

This is  well known but we
decided to write it down for readers
less familiar with the classical subject
of Cremona Transformations.
The map is the standard
Cremona transformation centered on
$\{p_1,p_2,p_3,p_6,\ldots,p_n\}$. Let
$$\omega^\K_5:\P^{n-3}\rat
\P^{n-3}=:\P $$
be the map given by the linear system
$$|\O_{\P^{n-3}}(n-3)\otimes(\otimes_{i\in\{1,2,3,6,\ldots,n\}}
\I^{n-4}_{p_i})|.$$
Let
\[
\xymatrix{
&Z\ar[dl]_{p}\ar[dr]^{q}
&\\
\P^{n-3}\ar@{.>}[rr]^{\omega^\K_5}&&\P}
\]
be the usual resolution obtained by
blowing up, in dimension increasing order,
all linear spaces spanned by points in
$\{p_1,p_2,p_3,p_6,\ldots, p_n\}$.
Let $l\subset\P$ be a general line then
$q^{-1}$ is well defined on $l$ and
$p(q^{-1}(l))$ is a rational normal curves passing
through
$\{p_1,p_2,p_3,p_6,\ldots,p_n\}$. Let
$E_i\subset Z$ be the exceptional
divisor corresponding to the blow up of the
points $p_i$. Then we have $q^{-1}(l)\cdot E_i=1$,
for $i\in\{1,2,3,6,\ldots,n\}$. While
 $q^{-1}(l)\cdot F$ vanishes for any other
$p$-exceptional divisor.
This description allows us to easily
compute  the degree of $\L_5$
$$\deg\L_5=\deg\omega_5^\K(\L)=(n-3)-\sum_{h\neq 5}\mult_{p_h}\L=2$$
This yields
$\omega_5^\K(\L)\subset|\O(2)|$.

To complete the analysis  we have to
understand the base locus of this system
of quadrics.
Let $\K_5$ be the Kapranov set labelled by
$\{1,2,3,4,6,\ldots,n\}$.

In our convention, see Definition
\ref{def:cre}, for $\{i,j\}\subset\{1,2,3\}$
we have
$$\omega^\K_5(V^4_{i,j})=V^5_{\{i,j,4\}^*\setminus\{5\}}$$
That is this line is sent to
a codimension two linear space.
Let $P_{h}=\omega^\K_5(V^4_{i,j})$, for
$\{i,j,h\}=\{1,2,3\}$.
The general
element in $\L$ intersects a general
point of $V^4_{i,j}$ and therefore its
transform via $\omega^\K_5$ has to
contain $P_h$. The
hypothesis  $5\in
I$ tell us that the map
$\omega^\K_5$ is well defined on the
general point of $P=V^4_I$.
Let $P_4=\omega^\K_5(P)$, and $S:=\omega^\K_5(\langle
p_1,p_2,p_3\rangle)$. Then $P_4$ has to be
contained in $\Bl\L_5$ and
$$S=\cap_{i=1}^4 P_i.$$
In conclusion we have:
\begin{itemize}
\item[.]  $\Bl\L_4=\cup_{i=1}^4 P_i\supset\K_5$
\item[.]  $\Sing(\L_5)=S$.
\end{itemize}\label{con:q}

Hence the linear
system $\L_5$ is a pencil of quadrics
with four codimension two linear spaces
in the base locus. That is the cone, in $\P^{n-3}$, over
a pencil of conics through 4 general
points and vertex $V^5_{I\setminus\{5\}}$.

To study fiber type morphisms from
$\Mo{n}$ it is important to control the
base locus of $M_\K$-linear
systems. The easiest base loci are those
of forgetful maps
$\phi_I:\Mo{n}\to\Mo{r}$. For this we
introduce the following definitions.
\begin{definition} Let
  $\pi_i:\P^{r-2}\to\P^{r-3}$ be the
  projection from a Kapranov point
  $p_i$, and
  $\L=|\O_{\P^{r-2}}(1)\otimes\I_{p_i}|$. Define
$${\mathcal
    C}^i_{r-3}:=\omega^\K_i(\L)\subset|\O_{\P^{r-2}}(r-2)|,$$
  to be the transform of
  hyperplanes through the point $p_i$.
We say that an $M_\K$-linear system
$\M$ on $\P^{n-3}$, has
base locus of type $\Phi_r$ if $\Bl\M$
is either a codimension $r-2$ linear
space or the cone over $\Bl{\mathcal
    C}^i_{r-3}$ with vertex a linear
space of codimension $r-1$. Equivalently $\M$ has base locus of type $\Phi_r$ if it is an $M_\K$-linear system
with linear base locus of codimension $r-2$ up to standard Cremona transformations. 
\end{definition}

In this notation ${\mathcal C}^i_1$ is a pencil
of plane conics through $4$ fixed points. The
above construction shows that to a
forgetful map $\phi_I:\Mo{n}\to\Mo{4}$
are associated $M_\K$-linear systems
with base locus of type $\Phi_1$. This is actually the main motivation of our definition.
 We
will use, and improve this observations first
in Proposition \ref{pro:C} and further
in Lemma \ref{lem:connect}. The main
point in our construction is that the
base locus of $M_\K$-linear system is
enough to characterise linear systems
inherited by forgetful maps.

The special case of forgetful maps onto
$\P^1$ is the one we use in this
paper. Nonetheless we would like to
stress that a similar behaviour applies
to an arbitrary forgetful map onto
$\Mo{r}$, for $r<n$.

\begin{proposition} Let
  $\phi_I:\Mo{n}\to\Mo{r}$ be a
  forgetful morphism. Assume that $1\in
  I$ and let
$$\xymatrix{
\overline{M}_{0,n}\ar[d]_{f_1}\ar[rr]_{\phi_I}&&\Mo{r}\ar[d]^{f_i}\\
P^{n-3}\ar@{.>}[rr]_{\pi_I}&&\P^{r-3}}
$$
be the usual diagram.
Then $\pi_I$ is given by a sublinear
system $\L_1\subset|\O(r-2)|$, the
general fiber of $\pi_I$ is  a cone, with vertex $V^1_{I\setminus\{1\}} \subset
\P^{n-3}$, over a rational
normal curve of degree $n-2-|I|$, and
$\L_1$ has Base locus of type $\Phi_r$.\label{pro:C}
\end{proposition}
\begin{proof} The morphism $\Phi_I$ can
  be factored as follows
\[
\xymatrix{
\overline{M}_{0,n}\ar[d]^{\phi_{I}}\ar[r]^{\phi_{I\setminus\{2\}}}
&\Mo{r+1}\ar[dl]^{\phi_2} \\
\Mo{r}&}
\]
Hence we have the induced diagram
\[
\xymatrix{
\overline{M}_{0,n}\ar[dr]_->>>>>{f_2} \ar[d]^{\phi_{I}}\ar[r]^{\phi_{I\setminus\{2\}}}
&\Mo{r+1}\ar[dr]^{f_2}\ar[dl]_-<<<<{\phi_2} \\
\Mo{r}\ar[dr]^{f_r}&\P^{n-3}\ar[d]^{\pi_{I}}\ar[r]^{\pi_{I\setminus\{2\}}}&\P^{r-2}\ar[dl]^{\f}\\&\P^{r-3}&}
\]
where $\pi_{I\setminus\{2\}}$ is a
linear projection and $\f$ is the map induced by
${\mathcal C}^2_{r}$. The claim then follows.
\end{proof}

\section{Base point free pencils on $\Mo{n}$}

A base point free pencil $\L$ on $\Mo{n}$ is the datum of a couple $(L,V)$
on $\Mo{n}$, where $L$ is a line bundle on $\Mo{n}$
and $V \subset H^0(\Mo{n}, L)$ is a two-dimensional subspace. The natural map $V \otimes \O_{\Mo{n}} \to L$ is surjective and this datum
is equivalent to a surjective morphism $f: \Mo{n}\to\P^1$ such that $L=f^{*}(\O(1))$.

Fix a Kapranov map
$f_i:\Mo{n}\to\P^{n-3}$ and let $\L_i:=f_*\L$. Then the linear system $\L_i$
is an $M_\K$-linear system inducing a birational map
$\pi^i:\P^{n-3}\rat\P^1$. Moreover
$\L=f_{i*}^{-1}\L_i$ and $f$ is a, not necessarily
minimal, resolution of the indeterminacy
of the map $\pi^i$.
To the map $f$ we therefore  associate a diagram

$$\xymatrix{
\overline{M}_{0,n}\ar[d]_{f_i}\ar[rr]_{f}&&\P^1\ar[d]^{\iso}\\
\P^{n-3}\ar@{.>}[rr]_{\pi^i}&&\P^{1}}
$$
where $\pi^i:= f\circ f_i^{-1}$.
The rational map $\pi^i$ is uniquely associated to a
pencil $\L_i \subset |\O_{\P^{n-3}}(d_i)|$
free of fixed divisors on $\P^{n-3}$.
We have $\L_i=(\O_{\P^{n-3}}(d_i), W_i)$
where $W_i \subset H^0(\P^{n-3}, \O(d_i))$ has dimension
two, and any element of $V$ is the strict transform
of an element of $W_i$, i.e. the strict transform
map $f_{i*}^{-1}: W_i \to V$ is an isomorphism
and the support of the cokernel of the evaluation map
$$ev_i : W_i \otimes \O_{\P^{n-3}} \to \L_i$$
on $\P^{n-3}$ does not have divisorial components.
In particular the support of the cokernel of $ev_i$
is the base locus $\Bl\L_i$ of $\L_i$. \par \noindent

The most important example of dominant maps
$f: \Mo{n} \to \P^1$ is given by the forgetful
maps already described in Section \ref{sec:bir}.
The goal of this section will be to prove that
in fact any surjective map with
connected fibers $f: \Mo{n} \to \P^1$
is in fact a forgetful map.
The criterion we are going to use in order to understand
whether a morphism $f: \Mo{n} \to \P^1 $
is a forgetful map is the following:

\begin{proposition} Let $f:\overline{M}_{0,n}\to X$ be
a surjective morphism. Let
  $A\in Pic(X)$ be a base point free linear
  system and $\L_i=f_{i*}(f^*(A))$. Assume that for
  some $j$ $\mult_{p_j}\L_i=\deg
  \L_i$. Then $f$ factors through the
  forgetful map $\phi_j:\Mo{n}\to\Mo{n-1}$.

Let $f:\Mo{n}\to\Mo{r}$ be a surjective morphism
 and $\pi:\P^{n-3}\rat\P^{r-3}$ the induced map. Let
$\L_i=f_{i*}(f^*(g_{j*}^{-1}(\O(1))$ and assume
that
$\L_i\subset|\O_{\P^{n-3}}(1)|$. Then $f$
is a forgetful map.
\label{pro:ind}
\end{proposition}

\begin{proof} Let
  $\phi_j:\Mo{n}\to\Mo{n-1}$ be the
  forgetful morphism. Fibers of $\phi_j$ are mapped
  by $f_i$ to lines in $\P^{n-3}$ through the point $p_j$.
  If $\mult_{p_j}\L_i=\deg \L_i$ the restriction of $\L_i$
  to each such line has a trivial moving part. In particular
  the linear system $f^*(A)$ is base point free and
  numerically trivial on every fiber of $\phi_j$. Moreover
  $Pic(\Mo{n}/\Mo{n-1})=Num(\Mo{n}/\Mo{n-1})$,
  therefore $f^*(A)$ is $\phi_j$-trivial.
  This shows that fibers of $\phi_j$ are contracted by $f$.
  For any $h \ne i, j$ the map $\phi_j$ has a section
  $s_{j,h}: \Mo{n-1} \to E_{j,h} \subset \Mo{n}$, described
in Section \ref{sec:notation} and then a morphism $g:=f \circ s_{j,h}: \Mo{n-1} \to X$ is given
  such that $f=g \circ \phi_j$. Notice that $g$ does not depend
  on the choice of $h \ne i,j$.

Assume that
$\L_i=|\O_{\P^{n-3}}(1)\otimes\I_{V^i_I}|$ for some
vital cycle $V^i_I \subset \P^{n-3}$. Then for any vital
point $p_k \in V^i_I$ we have $\mult_{p_k}\L_i=\deg \L_i$. Then by the first statement
we have that $\phi_k$ factors $f$ for any $k\in I$. Then there is a map
$g:\Mo{r}\to\Mo{r}$ such that $f=g\circ\phi_I$. Let $\gamma:
\P^{r-3}\map\P^{r-3}$ be the induced
map. Then $\gamma$ is associated to a
linear system of hyperplanes and it is
therefore a 
projectivity that eventually permutes the Kapranov set on $\P^{r-3}$, keep in mind our Definition \ref{def:forg}.
\end{proof}

\begin{definition}\label{def:str}
Let $\L_i:=(L_i, W_i)$ be an $M_\K$-linear system
on $\P^{n-3}$ and $A_i\in\L_i$ a general
element.
Let $H=H^{i\vee}_{h,k}$ be a vital hyperplane.
We say that the restriction of $\L_i$ to $H$
is dominant if the restriction map
$res_H: W_i  \to H^0(H, L_{i|_H})$
is injective.
Let $p_j \in \K$ be a Kapranov point. Let
$\epsilon_j: Y_j \to \P^{n-3}$ be the blow-up
of $p_j$ with exceptional divisor
$E_j$. Assume that $\epsilon_j^*
A_i=A_{iY}+mE_j$. Then
$$\L_{i,Y_j}:=(\epsilon_j^*L_i-mE_j,W^Y_i)$$
is  the strict transform of $\L_i$.
We say that $\L_i$ is dominant at the
first order of $p_j$
if the pullback map
$\epsilon_j^{*}: W_i \to H^0(E_j,(\epsilon_j^* L_i-mE_j)_{|E_j})$
is injective and $\L_{i,Y_j|E_j}$ is without fixed divisors. \label{def:1st}
\end{definition}

We sketch here the ideas underlying
our argument in order to characterise
base point free pencils on $\Mo{n}$.\par \noindent
We  proceed by induction on $n$ and assume that
all base point free pencils on $\Mo{n-1}$ inducing
a surjective map with connected
fibers $f: \Mo{n-1} \to \P^1$  are
forgetful maps. The well known  case of
$n=5$ is the beginning of the induction argument.
According to our criterion, Proposition \ref{pro:ind}, for the induction step 
it is  enough to show that there exists a Kapranov map
$f_i$ and a Kapranov point $p_j$ such that
$$\mult_{p_j}\L_i=\deg\L_i.$$
To produce this point we find a vital
hyperplane $H$ such that the restriction of $\L_i$
to $H$ is dominant. Then the hyperplane $H$  has
a Kapranov set and it is the image under a Kapranov map
of $\Mo{n-1}$, see remark \ref{rem:ka}. By induction we may find the required point for $\L_{i|H}$ and then lift it to the linear system
$\L_i$.
Here is a list of concerns in applying this idea:
\begin{itemize}
\item[.] How can we find a vital hyperplane $H$ such that the restriction of
 $\L_i$ to it is dominant ?
\item[.] How to compare $\mult_{p_j}\L_{i|H}$ with $\mult_{p_j}\L_i$ ?
\item[.] What  if the restricted
morphism has either non connected
fibers or  fixed components ?
\end{itemize}
As a matter of fact, even if $(L_i,
W_i)$ is free of fixed divisors and if
it induces
a map with connected fibers, this may
not be the case for the restricted
linear system on any vital hyperplane $H
\subset \P^{n-3}$.  Keep in mind that
there are only finitely many of those.

The desired hyperplane $H$ is produced
in Lemmata \ref{step1} and
\ref{step2}. The basic idea is that the
base locus of
$\L_i$ cannot be empty because
$\P^{n-3}$ does
not carry base point free pencils so that
there exists some point $p_j$
contained
in the base locus of $\L_i$. Notice that
this does not mean that such a point is  
an isolated component of
$\Bl\L_i$. Thanks to
Lemma \ref{lem:bs} we can prove that the $M_\K$-linear
system $(L_i, W_i)$ is dominant
at the first order of $p_j$.  To apply
induction on the exceptional divisor
over the point $p_j$ we have
to study pencils with possibly non
connected fibers. This is done in Lemma
\ref{lem:connect}. With this and
induction hypothesis we know that the
pencil induced on the exceptional
divisor has base locus of type
$\Phi_1$. Hence we may apply
induction  and find a hyperplane $H$ such that
$\L_i$ restricted to $H$ is dominant.

Finally, we use Lemma \ref{fixed} in
order to show that we can in fact
exclude
the presence of fixed divisors on the
restricted linear systems, so that we
can really
infer properties of $(L_i, W_i)$ from
properties of the restriction to some
hyperplane $H$.

We now prove the above mentioned Lemmata.
Let us fix a pencil $\L_i$, without fixed components, together with
the usual diagram
$$\xymatrix{
\overline{M}_{0,n}\ar[d]_{f_i}\ar[rr]_{f}&&\P^1\ar[d]^{\iso}\\
\P^{n-3}\ar@{.>}[rr]_{\pi^i}&&\P^{1}}
$$
and notation. Let $p_j\in\K\cap\Bl\L_i$
be a point
and $\epsilon_j: Y_j \to \P^{n-3}$  the blow-up
of $p_j$ with exceptional divisor $E_j$.
 Let
$\L_{i,Y_j} = \epsilon_j^{*}\L_i -m_jE_j$ be
the strict transform of $\L_i$, for some
positive $m_j$.

\begin{lemma}
The linear system $\L_i=(L_i, W_i)$
is dominant at the first order of
$p_j$ and for any $A_1,A_2 \in
\L_i$ we have
$$\mult_{p_j}A_1=\mult_{p_j}A_2.$$
Let $H=H^{i\vee}_{hk}$ be a vital
hyperplane containing $p_j$ and
 $A\in\L_i$ a general element. Then we have
$$\mult_{p_j}A=\mult_{p_j}A_{|H} $$
\label{step1}
\end{lemma}
\begin{proof}
In the above notation we know, by
Lemma \ref{lem:bs}, that
$\L_{i,Y_j|E_j}$ has not fixed
components. Hence
 the image of the pullback map
$\epsilon_j^{*}: W_i \to H^0(E_j,\L_{i,Y_j|E_j})$
is not one dimensional. Since $\dim
W_i=2$ we conclude that
$\epsilon_j^{*}$ is injective as
required. The injectivity of
$\epsilon_j^{*}$ forces every element in
$\L_i$ to have the same multiplicity at
$p_j$.

Let $H_{Y_j}$ be the strict transform of
$H$ on $Y_j$. Then again by Lemma
\ref{lem:bs} we know that
$\Bl\L_{i,Y_j}\not\supset E\cap
H_{Y_j}$. Therefore the general element
$A\in\L_i$ satisfies
$\mult_{p_j}A=\mult_{p_j}A_{|H}$.
\end{proof}

\begin{lemma}
Let $H=V^i_I\subset \P^{n-3}$
be a vital hyperplane such that $p_j \in H$.
If $ f(E_{i,I})$ is a point and if $H_j$ is
the strict transform of $H$ under $\epsilon_j$,
then $L_{i,Y_j|E_j}$ is trivial along $H_j \cap E_j$.
\label{step2}
\end{lemma}
\begin{proof}
The morphism $f$ is a resolution of indeterminacies of $\pi^i$ and $f_i$
factors through $\epsilon_j$. Then we have the result.
\end{proof}

In Section \ref{sec:bir} we proved that
every forgetful map onto $\P^1$ induces
an $M_\K$-linear system of degree at
most 2 with base locus of type $\Phi_1$.
One cannot expect that all
morphisms to $\P^1$ have bounded
degree. On the other hand, under
suitable hypothesis, the base locus
of $M_\K$-linear systems is unaffected by
connectedness of fibers.

\begin{lemma}\label{lem:connect} Assume
  that every dominant morphism
  \hbox{$g:\Mo{n}\to \P^1$} with connected
  fibers is a forgetful map. Then
  $\Bl\L_i$ is of type
  $\Phi_1$.
If moreover $n\geq 7$  there are vital points $p_j$  satisfying  $\mult_{p_j}\L_i=\deg\L_i$.
\end{lemma}
\begin{proof}
Let $h:\Mo{n}\to C$ be the Stein
factorization of $f$. Let
$\nu:\tilde{C}\to C$ be the
normalization. Then there is a unique
map
$f':\Mo{n}\to\tilde{C}$ such that
$h=\nu\circ f'$.
The variety $\Mo{n}$ is rational,
therefore $\tilde{C}\iso\P^1$ and
$|f^*\O(1)|\subset |f^{\prime*}\O(\gamma)|$ for
some integer $\gamma$.
By
hypothesis $f'$ is a forgetful map
therefore we may choose $i$ in such a
way that
$|f_{i*}f^{\prime*}\O(1)|\subset|\O_{\\P^{n-3}}(1)|$
satisfies
$$\Bl|f_{i*}f^{\prime*}\O(1)|=V^j_I,$$
where $V^j_I$ is
a codimension two irreducible vital
space. Then the elements
in the linear system
$\L_i$ are union of $\gamma$ hyperplanes
containing $V^j_I$ and
$$\Bl\L_i:=\Bl|f_{i*}f^{*}\O(1)|=\Bl|f_{i*}f^{\prime*}\O(1)|=V^j_I.$$ To conclude it is enough
to apply standard Cremona
transformations to this configuration as
described in Section \ref{sec:bir}.
In particular if $n\geq 7$ all linear systems of type $\Phi_1$ are cones with non empty vertex and therefore there is at least a point $p_j$ with
$\mult_{p_j}\L_i=\deg\L_i$.
\end{proof}

We conclude this technical part taking
into account the eventual fixed
divisors.

\begin{lemma}
Let $H=H^{i\vee}_{h,k}$ be a vital hyperplane in $\P^{n-3}$, with $j\neq h,k$.
Assume that $\L_i=(L_i, W_i)$ has a dominant restriction to $H$.
Assume that $F$ is the fixed divisor of the restricted system
$\L_{i|H}$ and that $p_j \not\in F$. Then
$$F= \langle p_l | l \neq h,k,j \rangle =V^{i}_{\{i,h,k,j\}^*}.$$
\label{fixed}
\end{lemma}
\begin{proof}
The fixed divisor  $F\subset H$ is in the base
locus of $(L_i, W_i)$. By hypothesis
$\L_i$ has not fixed divisors. Then the support of $F$ must be an irreducible component
of $\Bl\L_i$. In particular $F$
does not contain $p_h, p_k$, and therefore
cannot intersect the line
$\langle p_h,p_k\rangle=V^{i}_{h,k}$. By hypothesis
we have $p_j \not\in F$. Hence the only
possibility left is  $F= \langle p_l | l \neq h,k,j \rangle $.
\end{proof}

Our inductive argument starts with the classical case of $\Mo{5}$. The next step is a special case of \cite{FG} study of the cone of effective curves of $\Mo{6}$. From our point of view the $n=6$ case is a bit more complicate because in this case it is not true that
there is always a point with $\mult_{p_j}\L_i=\deg\L_i$. The pencil ${\mathcal C}^i_1$ is a pencil
of plane conics. Nonetheless we prefer to prove the $n=6$ case with
our techniques to show a slightly more complicate application of Cremona transformations $\omega^\K_h$.
Suppose $\L=(L, V)$ is a pencil on $\Mo{6}$.
Choose a Kapranov map $f_i$ and the
induced pencil $\L_i=(L_i, W_i)$ on
$\P^3$. From Lemma \ref{step1}
we find a point $p_j \in \K\cap \Bl\L_i$ such that
$\L_i$ dominates $p_j$ at first order.
Let $E_j$ be the exceptional divisor of
the blow-up of $p_j$. This is a plane
with a natural Kapranov set induced by
the lines in $\P^3$ joining $p_j$ with
any other point of $\K \subset \P^3$.
Furthermore the strict transform
$\L_{i,Y_j}$ is in a natural way an $M_\K$-linear system. We may apply
Lemma \ref{lem:connect} in order to
deduce
that its base locus is of type $\Phi_1$.
After possibly switching the Kapranov
map, see Remark
\ref{rem:kcre},
we may assume that
$\Bl\L_{i,Y_j}$ consists of one point, say $p_l$, and that
fibers of the rational map to $\P^1$ are
sets of lines
through $p_l$. This,
together with Lemma \ref{step2} yields
that if $ H \subset \P^3$ is a vital
hyperplane containing $p_j$ and not
containing $p_l$,
the restriction of $(L_i, W_i)$ to $H$
is dominant.
We may assume that $j=2$, $l=1$, and $i=6$.
Let $H_1=H^{6\vee}_{1,3}$ and $H_2=H^{6\vee}_{1,4}$ be vital hyperplanes. By construction the restriction of $\L_6$ to $H_s$ is dominant, for $s=1,2$.

\begin{claim} We may assume that 
 $(L_{6|H_s},
W_6)$ is free of fixed divisors for $s=1,2$.
\end{claim}
\begin{proof}[Proof of the Claim]
 Consider
$H_1=H^{6\vee}_{1,3}$,
$H_2=H^{6\vee}_{1,4}$, and
$H_3=H^{6\vee}_{1,5}$. Then the restriction of $\L_6$ to $H_s$ is dominant, for $s=1,2,3$.
Assume that for any pair of $H_s$, the linear system $\L_{6|H_s}$ has fixed divisors, say $F_s\subset H_s$.
By construction $\Bl\L_{6,Y_2}$ is a single point corresponding to the line $V^6_{1,2}$, hence $p_2\not\in
F_s$. Then  Lemma
\ref{fixed} yields
\begin{equation}
\Bl\L_6\supseteq V^6_{4,5}\cup V^6_{3,5}\cup V^6_{3,4}.
\label{eq:6}
\end{equation}
By Lemma \ref{step1} $\L_6$ is
 dominant at the first order at any Kapranov point
and by equation (\ref{eq:6})
there are at least two points in
$\Bl\L_{6,Y_h}\cap E_h$, for $h=3,4,5$. Let $\K_h$ be
the Kapranov set induced on $E_h$, then via Lemma
\ref{lem:connect} we conclude that
\begin{equation}
\label{eq:plane}
\Bl\L_{6,Y_h}\supset \K_h\ \ \mbox{ for $h=3,4,5$.}
\end{equation}
Let $\M$ be the movable part of the linear system $\L_{6|H_s}$.
Let $\M_{h,Y_h}$ be the strict transform of $\M$ on the blow up $Y_h$. Then by Lemma \ref{lem:connect}
$\Bl\M$ is of type $\Phi_1$, therefore $\M_{h,Y_h}\cap E_h=\emptyset$.
This together with equation (\ref{eq:plane}) yields that every vital line contained in $H_s$ and passing through $p_h$
is a fixed component of $\L_6$. In particular we derive the contradiction
$$V^6_{2,h}\subseteq\Bl\L_6$$
\end{proof}

If $\Bl\L_{6}\cap H_s$ is a single point $p_h$, then $\mult_{p_h}\L_{6|H}=\deg\L_{6|H}$. Hence by Lemma \ref{lem:connect}, and \ref{step1} we know that
$$\mult_{p_h}\L_6=\mult_{p_h}\L_{6|H}=\deg \L_{6|H}=\deg\L_6$$ so that $$\mult_{p_h}\L_6=\deg\L_6$$
and we conclude by
Proposition \ref{pro:ind} that
$f$ factors via the forgetful map $\phi_h$. Then every map, with connected fibers, from $\Mo{5}$ is a forgetful map and we conclude that
$f$ itself is a forgetful map.

Assume that $\Bl\L_{6}\cap H_s$ is the full Kapranov  set for $s=1,2$. Then consider the Cremona Transformation $\omega^\K_5$.
Let $H_s^{\prime}=\omega^\K_5(H_s)$, for $s=1,2$.

\begin{claim} We may assume that the restricted linear system $\L_{5|H_1^{\prime}}$ has not fixed divisors.
\end{claim}
\begin{proof}[Proof of the Claim] Assume that $\L_{5|H_s^{\prime}}$ has fixed divisors, for $s=1,2$.
The linear system $\L_6$ is dominant at the first order in any Kapranov point,
 then the only possible fixed divisor of $\L_{5|H^{\prime}_s}$ is
$V^5_{2,5-s}$. This forces, as in the previous Claim,
$$\Bl\L_{5,2}\supset\K_2, $$
where $\K_2$ is the Kapranov set induced on the exceptional divisor $E_2$. Then as before we have
$$\Bl\L_5\supset  V^5_{2,6},$$
a contradiction.
\end{proof}
The restriction $\omega^\K_{5|H_s}$ is a standard Cremona transformation of $\P^2$. Hence the claim shows that
$\mult_{p_6}\L_{5|H_1^{\prime}}=\deg\L_5$, and we conclude as above that $f$ factors through a forgetful map.
This concludes the $n=6$ case. We are ready for the proof of the following:
 
\begin{theorem} Let
  $f:\overline{M}_{0,n}\to
  \P^1\iso\overline{M}_{0,4}$ be a non constant
  morphism with connected fibers. Then $f$ is a forgetful map.
\label{thm:p1}
\end{theorem}
\begin{proof} We prove the claim by
  induction on $n$. We already discussed
  the $n\leq 6$ case.
  We have
  to prove the result for $n\geq 7$. Let
  $f_n:\overline{M}_{0,n}\to\P^{n-3}$ be
  a Kapranov map with
  $\K=\{p_1,\ldots,p_{n-1}\}$.
Let $\L_n=(L_n,V_n)=f_n(f^*(\O(1)))$ be the associated linear system.
Then the linear system  $\L_n$ is a
pencil of hypersurfaces, without fixed components, say
$\L=\{A_{1},A_{2}\}$, and it is an $M_\K$-linear system.

>From Lemma \ref{step1} we know that there exists $p_1 \in \K\cap\Bl\L_n$
such that $\L_n$ is dominant at the first order of $p_1$.
>From Lemma \ref{step1} we get:
\begin{equation}
\label{eq:mult}
m=\mult_{p_1}A_{1}=\mult_{p_1}A_{2}
\mbox{ and }
\mult_{p_1}\L=\mult_{p_1}\L_{|H^{n\vee}_{h,k}}\mbox{   for $h,k\neq 1$.}
\end{equation}

Let $\epsilon_j:Y_j\to \P^{n-3}$ be the blow
up of $p_j$ with exceptional divisor
$E_j$, for $j=1,\ldots n-1$.
By induction and Lemma \ref{lem:connect}, after possibly passing to another Kapranov map,
we may assume that
$\Bl\L_{n,Y_1}$
is a codimension two
linear space, and $\epsilon_1(\Bl\L_{n,Y_1})=V^n_{1,\ldots,n-4}$.

>From Lemma \ref{step2} we deduce that if $H$ is a vital hyperplane containing $p_1$ but not
containing $V^n_{1,\ldots,n-4}$, the
restriction of $\L_n$ to $H$ is
dominant.

\begin{claim}
We may choose $H$ in such a way that
$\L_{n|H}$ is free of fixed divisors.
\end{claim}
\begin{proof}[Proof of the Claim]Let $H_1=H_{n-4,n-3}^{n\vee}$ and
$H_2=H_{n-5,n-2}^{n\vee}$ be two vital
hyperplanes. Assume that the restriction of
$\L_n$ to $H_i$ has a not empty fixed
divisor $F_i$. Then  from Lemma \ref{fixed}
the support of $F_1$ and of $F_2$ are
respectively
$V^n_{2,\ldots,n-5,n-2,n-1}$ and
$V^n_{2,\ldots,n-6,n-4,n-3,n-1}$.

Lemma \ref{step1}, applied to the point
$p_{n-1}$, yields
that $\L_n$ dominates $p_{n-1}$ at first
order. Let $\K_{n-1}$ be the induced
Kapranov set on $E_{n-1}$. Then
$\Bl\L_{n,E_{n-1}}\cap{E_{n-1}}$ has  two
irreducible components meeting in codimension
$4$. On the other hand by Lemma
\ref{lem:connect} $\Bl\L_{n,E_{n-1}}\cap{E_{n-1}}$ is
of type $\Phi_1$, a contradiction.
\end{proof}

Let $H$ be a vital hyperplane such that the restriction of $\L_n$ to $H$ is dominant and $\L_{n|H}$ is free of fixed divisors.
Then by Lemma  \ref{lem:connect} and Lemma \ref{step1} there is a Kapranov point $p_h\in H$ such that
$$\mult_{p_h}\L_n=\mult_{p_h}\L_{n|H}=\deg \L_{n|H}=\deg\L_i$$ so that $$\mult_{p_h}\L_i=\deg\L_i$$
We conclude by
Proposition \ref{pro:ind} that
$f$ factors via the forgetful map $\phi_h$.
That is $f=g\circ\phi_h$ for some morphism, with connected fibers, $g:\Mo{n-1}\to\P^1$.
By induction hypothesis $g$ is a forgetful map. Henceforth  $f$ is forgetful.
\end{proof}

As a further attempt to make the result
and its main ideas clearer we add a proof of Theorem
\ref{thm:p1}  kindly suggested by James M\textsuperscript{c}Kernan
and Jenia Tevelev, \cite{MT}. We thank
James and Jenia for their help in
translating our projective arguments into
a better known dictionary and also to
produce a proof that shows at best the
``almost toric'' nature of $\Mo{n}$.

\begin{proof}[Proof of Theorem \ref{thm:p1}\ {\rm (\cite{MT})}] We proceed by induction on $n$.  We may assume that $n\geq 7$.
Let $f_{ij}$ be the restriction of $f$ to $\delta_{ij}:=E_{i,j}\simeq\Mo{n-1}$.  There are two cases:
\begin{enumerate} 
\item $f_{ij}$ is never constant.  
\item $f_{ij}$ is constant, for at least one pair $\{i,j\}$.  
\end{enumerate} 
 
Suppose we have (1).  We will derive a
contradiction.  By induction, we know
that each $f_{ij}$  is a composition of
a forgetful map $f'_{ij}\colon\,\delta_{ij}\to\P^1$ and a finite  
morphism $g_{ij}\colon\P^1\to \P^1$. 
Notice that forgetful maps $f'_{ij}$ and $f'_{kl}$ agree on intersections 
$\delta_{ij}\cap\delta_{kl}$ each time these divisors have a non-empty intersection, i.e.~when 
$\{i,j\}$ and $\{k,l\}$ do not contain common elements. Indeed, both $f'_{ij}$ and $f'_{kl}$ 
restrict to some forgetful maps $\delta_{ij}\cap\delta_{kl}\simeq\Mo{n-2}\to\Mo{4}\simeq\P^1$.
But 
$$(f'_{ij})^*{\O}_{\P^1}(a)\simeq (f_{ij})^*{\O}_{\P^1}(1)\simeq (f_{kl})^*{\O}_{\P^1}(1)\simeq (f'_{kl})^*{\O}_{\P^1}(b)$$
for some positive integers $a$ and $b$ and 
and a forgetful map $\Mo{n-2}\to\Mo{4}\simeq\P^1$ 
is uniquely determined by the pull-back of ${\O}_{\P^1}(1)$ (up to a multiple).

There are two cases, up to the obvious symmetries, 
$$
f'_{12}=\begin{cases} \pi_{3,4,5,6} \\
                     \pi_{\{1,2\},3,4,5}.
\end{cases}
$$
Consider $f_{67}$.  Up to even more symmetries, we must have
$$
f'_{67}=\begin{cases} \pi_{3,4,5,\{6,7\}} & \text{if $f'_{12}=\pi_{3,4,5,6}$} \\
                     \pi_{2,3,4,5}       & \text{if $f'_{12}=\pi_{\{1,2\},3,4,5}$.}
\end{cases}
$$
Possibly switching $\{1,2\}$ and $\{6,7\}$ we might as well assume that 
$$
 f'_{12}=\pi_{\{1,2\},3,4,5} \qquad
 \text{and} \qquad f'_{67}=\pi_{2,3,4,5}.
$$
It follows that $f$ restricted to both $\delta_{12}\cap \delta_{34}$ and $\delta_{34}\cap \delta_{67}$ is
constant.  But then $f'_{34}=\pi_{1267}$ and so $f'_{15}=\pi_{\{1,5\},2,6,7}$.  On the other
hand $f'_{15}=\pi_{\{1,5\},2,3,4}$, a contradiction.

So we must have (2). Assume that $f$ contracts, say, $\delta_{1n}$.  Let
$f_n\colon\Mo{n}\to \P^{ n-3}$ be the Kapranov map associated to the Kapranov set
$\{p_1,\ldots,p_{n-1}\}$.  That is, the map that blows up $n-1$ points
$\{p_1,\ldots,p_{n-1}\}$ in linear general position, and every linear space spanned by
these points.  Then $f_n$ contracts $\delta_{1,n}$ to the point $p_{1}$.  Let $\psi\colon
L_n\to {\P^{ n-3}}$ be the birational morphism which blows up every linear space blown up
by $\pi$, except those which contain $p_{1}$.  Notice that $L_n$ is a toric variety, and
there is a birational morphism $\f\colon\Mo{n}\to L_n$ which factors $f_n=\psi\circ\f$.
This yields an induced rational map $g=f\circ \f^{-1}\colon L_n\rat{\P^1}$.  Then the
rational map $g\colon L_n\rat \P^1$ is
\begin{itemize}
\item[(a)] regular at $p_1$;
\item[(b)] has a base locus of codimension 2 (as for any rational pencil);
\item[(c)] has a base locus contained in
  the indeterminacy locus of the   
birational map $\f^{-1}\colon L_n\rat \Mo{n}$. The latter is the union of linear subspaces  
passing through $p_1$ (in the Kapranov model).
\end{itemize}

It follows that the map $g\colon L_n\rat \P^1$
is actually regular. As for any morphism,
with connected fibers, to $\P^1$, 
it is given by a complete linear
series.  Therefore it is a toric   
morphism.
To conclude, we prove that it is one    
of the forgetful maps by studying the induced
map of fans. 

The fan $F_n$ of $L_n$ is obtained by taking the standard fan for $\P^{n-3}$ (with rays
$R_1,\ldots,R_{n-2}$, of which the first $n-2$ correspond to coordinate hyperplanes)
followed by its barycentric subdivision.  A toric morphism $L_n\to\P^1$ corresponds to a
linear map $g\colon \R^{n-3}\to\R$ that sends each cone of $F_n$ to either $\{0\}$, the
positive ray $\R_+$, or the negative ray $\R_-$.  We may assume without loss of generality
that $g(R_1)=\R_+$ and $g(R_2)=\R_-$.  The fan of $L_n$ contains the ray $C=R_1+R_2$ and
we should have $g(C)=0$.  Therefore, $g$ sends primitive generators of $R_1$ and $R_2$ to
opposite vectors $v_1$, and $v_2$ in $\R$.  We claim that $g(R_i)=0$ for $i>2$.  Assuming the
claim, we then have $v_1,v_2=\pm 1$ and the toric morphism is a resolution of the linear
projection from the intersection of the first two coordinate hyperplanes, which
corresponds to one of the forgetful maps.

Back to the claim, and arguing by contradiction, suppose that $g(R_3) =\R_+$.  Then
$g(-R_3)=\R_-$. But $-R_3$ is the barycenter of the top-dimensional cone of $F_n$ spanned
by all $R_i$ for $i\ne 3$. Then $F_n$ contains a cone with rays $R_1$, and $-R_3$, which
does not map to any cone.
\end{proof}

\begin{remark} It is interesting to note
  that the space $L_n$ is a moduli
space in its own right,  introduced by
Losev and Manin \cite{LM}, and the
morphism $\Mo{n}$ to $L_n$ has a natural modular interpretation.  
\end{remark}

\begin{corollary} Let
  $f\colon\overline{M}_{0,n}\to
  \P^1\iso\overline{M}_{0,4}$ be a non constant
  morphism. Then $f$ factors through a forgetful map and
  a finite morphism.
\label{cor:p2}
\end{corollary}
\begin{proof} Taking the Stein factorization and the normalization, as in Lemma
\ref{lem:connect}, we reduce the claim to Theorem \ref{thm:p1}.
\end{proof}

\section{Morphisms to $\overline{M}_{0,r}$}

In this final section we apply Theorem \ref{thm:p1} to deduce
that in fact any surjective morphism $f: \Mo{n} \to \Mo{r}$
is a forgetful map. As a corollary we compute the automorphisms group of $\Mo{n}$.

\begin{theorem}
Let $f:\overline{M}_{0,n}\to
\overline{M}_{0,r}$ be a
dominant morphism with connected fibers. Then $f$ is a
forgetful map.
\label{th:forg}
\end{theorem}
\begin{proof} We prove the Theorem by induction on $r$. The first step of the induction process is the content of Theorem \ref{thm:p1}.

Let us fix a Kapranov map $f_r: \Mo{r} \to \P^{r-3}$ and consider the
forgetful map
$$\phi_{r-1}:\overline{M}_{0,r}\to\Mo{r-1},$$
the Kapranov map $f_{r-1}: \Mo{r-1} \to \P^{r-4}$
and the projection
$\pi:\P^{r-3}\rat\P^{r-4}$ given by the linear system
  $\Lambda_{r-1}=|\O_{\P^{r-3}}(1)\otimes\I_{p_{r-1}}|$.
Then by induction hypothesis
 $(\phi_{r-1}\circ f):\overline{M}_{0,n}\to\Mo{r-1}$ is dominant
 and with connected fibers, hence a
  forgetful map.  This means that we can choose a Kapranov map
  $f_n: \Mo{n} \to \P^{n-3}$ such that
$$f_{n*}((f_r\circ
f)^{-1}_*(\Lambda_{r-1}))\subset|\O_{\P^{n-3}}(1)| .$$
Recall that, from Proposition \ref{pro:ind}, we obtain the thesis
if we show that
$$f_{n*}((f_r\circ
f)^*(\O_{\P^{r-3}}(1))\subset|\O_{\P^{n-3}}(1)| .$$
By construction we have $\Lambda_{r-1}=|\O_{\P^{r-3}}(1)\otimes\I_{p_{r-1}}|$
and $f_r^{-1}(p_{r-1})=E_{r,r-1}$. To conclude it is enough to
show that $f^{*}(E_{r,r-1})$ is $f_n-$exceptional.

The following  diagram, will help us along the proof,

\[
\xymatrix{
\overline{M}_{0,n}\ar[d]^{f_{n}}\ar[r]^f
&\overline{M}_{0,r}
\ar[d]^{f_r}\ar[r]^{\phi_{r-1}}&\Mo{r-1} \ar[d]^{f_{r-1}} &
                                                                \\
\P^{n-3}\ar@{.>}[r]^\f&\P^{r-3}\ar@{.>}[r]^\pi&\P^{r-4}
 }
\]

 Let $\L$ be the
linear system associated to the
 map $$(\pi\circ \f)=(f_{r-1}\circ \phi_{r-1}
 \circ f\circ f_n^{-1}).$$
By induction hypothesis
we may assume  that $f_{n}$ is such that $\L=|\O(1)\otimes\I_P|$,
where $P=\langle
p_{r-1},\ldots,p_{n-1}\rangle$. We fix notations in such a way that
$(\pi\circ\f)(p_j)=p_{j}$, for $j<r-1$, and
$\pi(p_j)=p_j$, for $j<r-1$.

For any $E_{j,r}\neq E_{(r-1),r}$ the
map $\phi_{r-1|E_{j,r}}:\Mo{r-1}\to\Mo{r-1}$
is a forgetful map onto $\Mo{r-2}$.
Then for any $E_{i,r}\subset\Mo{r}$, with
$i<r-1$, we have that
$$f^*(E_{i,r})=(\phi_{r-1}\circ f)^*(E_{i,r-1})=E_{i,n}$$
This shows that $f^*(E_{i,r})$ is
$f_n$-exceptional for $i < r-1$.

Notice that, once having fixed $f_r$ and chosen the forgetful
map $\phi_{r-1}$, we have found $f_n$ such that $f^{*}(E_{i,r})=E_{i,n}$
for $i < r-1$.

We are assuming $r\geq 5$ hence, once we fix the Kapranov map
$f_r:\Mo{r}\to\P^{r-3}$ there are at least 4
possible forgetful maps
$\phi_i:\Mo{r}\to\Mo{r-1}$, with $i <r$.
To any such $\phi_i$ we
may associate a Kapranov map
$f_{n_i}:\Mo{n}\to\P^{n-3}$ in such a
way that $f^*(E_{j,r})=E_{j,n_i}$, for $j\neq i$.
On the other hand, the divisor  $E_{i,j}\subset\Mo{n}$ is
sent to a point only by $f_i$ and
$f_j$.
Then we may assume that $n_1=n_2=n$.
 The image of divisors $E_{i,r}$
via $f_{n*}f^*$ does not depend on the
map $\phi_i$. Therefore $f_{n*}f^*(E_{i,r})$ is
a point for any $i=1,\ldots, r-1$.

By definition
$$f_r^*(\O(1))=f^{-1}_{r*}\Lambda_{r-1}+E_{(r-1),r}$$
hence we have
$$\L= f_{n*}((f_r\circ f)^*(\O(1)))=f_{n*}((f_r\circ
f)^*(\Lambda_{r-1}))\subset|\O_{\P^{n-3}}(1)|.$$
and $\f$ is
given by a linear system of
hyperplanes.
This is
equivalent to our statement by
Proposition \ref{pro:ind}
\end{proof}

This easily extends to morphisms onto
products of $\Mo{r_i}$.

\begin{corollary} Let $f:\overline{M}_{0,n}\to
  \overline{M}_{0,r_1}\times\ldots\times\overline{M}_{0,r_h}$ be a dominant fiber type morphism
with connected fibers. Then $f$ is a forgetful map.
\label{cor:main}
\end{corollary}
\begin{proof} It is enough to compose
  $f$ with the projection onto the factors.
\end{proof}

From
Theorem \ref{th:forg}
an automorphism of $\Mo{n}$
must preserve all forgetful maps. This
gives a very strong
condition on the induced linear system of
$\P^{n-3}$.
We are ready to prove the main result on $Aut(\Mo{n})$.
This is  classical for $n=5$.

\begin{theorem} Assume that $n\geq 5$. Then $\Aut(\Mo{n})=S_n$, the
symmetric group on n elements.
\end{theorem}
\begin{proof}
Let $g\in\Aut(\Mo{n})$ be
an automorphism. Let
$\phi_i:\Mo{n}\to\Mo{n-1}$ be the $i$-th
forgetful map. Then by
Theorem \ref{th:forg}
$g\circ\phi_i$
is a map forgetting an index
$j_i\in\{1,\ldots,n\}$.

This means that we can associate to $g$
the permutation $\{j_1,\ldots,j_n\}\in S_n$. Let
$\chi:\Aut(\Mo{n})\to S_n$ be the
associated map. The map
$\chi$ is a surjective morphisms. A
simple transposition is 
realized by the standard Cremona
transformations we recalled in
Definition \ref{def:cre}.

The main point is to determine the
kernel. Assume that
$\chi(g)=1$. That is
$g\circ\phi_i$ is forgetting the $i$-th
index for any
 $i\in\{1,\ldots,n\}$.
Fix a Kapranov map
 $f_n:\Mo{n}\to\P^{n-3}$. The
 automorphism $g$
induces
a Cremona transformation $\gamma_n$ on $\P^{n-3}$
that stabilizes the lines through the Kapranov points and also the rational normal curves
through $\K$.
Let
$\H_n\subset|\O(d)|$ be the
linear system associated to $\gamma_n$. Let $l_i\subset\P^{n-3}$ be a general line through $p_i$ and $\Gamma_n$ a general rational normal curve through $\K$. Then we have
$$\deg(\gamma_n(l_i))=d-\mult_{p_i}\H_n=1,$$
for any $i\in\{1,\ldots,n-1\}$, and
$$\deg(\gamma_n(\Gamma_n))=(n-3)d-\sum_{i=1}^{n-1}\mult_{p_i}\H_n=n-3.$$
These yield
$$n-3=(n-3)d-(n-1)(d-1)$$
and finally $d=1$. That is $\gamma_n$ is a projectivity that fixes $n-1$ points. Then $\gamma_n$ and henceforth $g$ are the identity.
\end{proof}

\end{document}